\documentclass[12pt]{article}
\usepackage{mathrsfs}
\usepackage{amsthm}
\usepackage{amssymb}
\usepackage{latexsym}
\usepackage{amsmath,amsfonts}
\usepackage{mathrsfs}
\usepackage{cases}
\usepackage{latexsym,bm}
\usepackage{indentfirst}
\usepackage{color}
\usepackage{ifpdf}
\usepackage{graphicx}
\usepackage{epstopdf}
\usepackage{epsfig}
\usepackage{psfrag}
\usepackage[
pdfauthor={lu},
pdftitle={Removable edges in cubic bricks},
pdfstartview=XYZ,
bookmarks=true,
colorlinks=true,
linkcolor=blue,
urlcolor=blue,
citecolor=blue,
bookmarks=true,
linktocpage=true,
hyperindex=true
]{hyperref}

\usepackage{graphicx}
\usepackage{color}

\title{ $b$-invariant  edges in  essentially 4-edge-connected near-bipartite cubic bricks \footnotetext{  The research
 is supported by National Natural Science Foundation of China
 (Grant No. 11671186).\newline E-mail addresses: flianglu@163.com
 (F. Lu), fengxing\_fm@163.com (X. Feng), wangyan@mnnu.edu.cn (Y. Wang).}}
  \author{ Fuliang Lu$^{a}$, Xing Feng$^b$,
Yan Wang$^{a}$\\ $^a$
\small {School of Mathematics and Statistics,
Minnan Normal University, Zhangzhou,  China}\\$^b$ \small {Faculty of Science, Jiangxi University of Science and Technology, Ganzhou, China}}
\date{}

\newtheorem{lem}{Lemma}
\newtheorem{thm}[lem]{Theorem}

\newtheorem{pro}[lem]{Proposition}
\newtheorem{conj}[lem]{Conjecture}

\begin{document}

\newcommand{\udots}{\mathinner{\mskip1mu\raise1pt\vbox{\kern7pt\hbox{.}}
\mskip2mu\raise4pt\hbox{.}\mskip2mu\raise7pt\hbox{.}\mskip1mu}}
\maketitle
\begin{abstract}

A {\em brick} is a  non-bipartite matching covered graph without non-trivial tight
cuts. Bricks  are
building blocks of matching covered graphs. We say that an
edge $e$ in a brick $G$ is {\em $b$-invariant} if $G-e$ is
matching covered and a tight cut decomposition of $G-e$ contains exactly one brick.
A 2-edge-connected cubic graph is  {\em essentially 4-edge-connected} if it
does not contain nontrivial 3-cuts. A brick $G$ is   {\em  near-bipartite} if it has a pair of edges
$\{e_1, e_2\}$ such that $G-\{e_1,e_2\}$ is bipartite and matching covered.

 Kothari,
de Carvalho, Lucchesi  and Little proved that each essentially
4-edge-connected cubic non-near-bipartite brick $G$, distinct
from the Petersen graph, has at least $|V(G)|$ $b$-invariant edges. Moreover,
they made a conjecture: every essentially 4-edge-connected cubic
near-bipartite brick $G$, distinct from $K_4$, has at least
$|V(G)|/2$ $b$-invariant edges. We confirm the conjecture in
this paper. Furthermore,
all the essentially 4-edge-connected cubic
near-bipartite bricks, the  numbers of   $b$-invariant edges  of which   attain  the lower bound, are presented.\\

\end{abstract}
\vskip 0.2in \baselineskip 0.1in
\section{Introduction}
All graphs considered in this paper are finite and may contain
multiple edges, but no loops.  We will generally follow the
notation and terminology used by Bondy and Murty in~\cite{BM08}. A
graph is called {\it matching covered} if it is connected, has
at least one edge and each of its edges is contained in some
perfect matching.
For the
terminology that is specific to matching covered graphs,
we follow Lov\'asz and Plummer~\cite{LP86}.

Let $G$ be a graph with vertex set $V(G)$ and edge set $E(G)$. For a vertex $v$, we use $N(v)$ to denote its neighborhood, that is $N(v):=\{u\in V(G): uv\in E(G)\}$.   For any
$X,Y\subseteq V(G)$,
denote by
$E_G(X,Y)$ the set of edges in $G$ with one end in $X$, the other in
$Y$. We say $\partial (X):=E_G(X,\overline{X})$ is an {\it edge cut} of $G$, where $\overline{X}:=V(G)-X$.
An edge cut
$C:=\partial (X)$ of $G$ is a {\it separating cut} if, for every edge $e$ in $G$, there exists a perfect matching $M_e$  that contains $e$ satisfying $|M_e\cap C|=1$, and is a {\it tight cut}
if $|C\cap M|=1$ for every perfect matching $M$ of $G$. Obviously,  every tight cut is separating.
An edge cut $\partial (X)$ is
{\it trivial} if either $|X|=1$ or $|\overline{X}| =1$. We call a
matching covered graph  free of  non-trivial tight cuts
  a {\it brick} if it is non-bipartite, and a {\it brace}
otherwise. Edmonds et al.~\cite{ELP8274} (also see Lov\'asz~\cite{L8772},
Szigeti\cite{S0259} and de Carvalho et al.~\cite{CLM18})
showed that a graph $G$ is a  brick if and only if $G$ is
$3$-connected and $G-\{x,y\}$ has a perfect matching for any two
distinct vertices $x, y\in V(G)$ (bicritical). Lov\'asz~\cite{L8772} proved
that any matching covered graph can be decomposed
into a unique list of bricks and braces by a
procedure called the tight cut decomposition.~In
particular, any two applications of the tight cut
decomposition of a matching covered graph $G$ yield
the same number of bricks, which is called the
{\it brick number} of $G$ and denoted by $b(G)$.

  A non-bipartite matching covered graph $G$ is {\em
near-bipartite} if it has a pair of  edges $e_1$
and $e_2$ such that $G-\{e_1,e_2\}$ is bipartite
and matching covered.
 An edge $e$ of a matching covered graph $G$ is
{\em removable} if $G-e$ is also matching covered.  Suppose $\{e_1,e_2\}\subseteq E(G)$.  We say that $\{e_1,e_2\}$ is a
{\em removable doubleton} of $G$ if neither $e_1$ nor $e_2$ are  removable, and  $G-\{e_1,e_2\}$ is matching covered. A removable
edge $e$ of  a matching covered graph $G$ is {\em $b$-invariant} if $b(G-e) = b(G)$, and   is {\em quasi-$b$-invariant}
if $b(G)=1$ and $b(G-e)=2$.
For a bipartite graph $G$, every removable edge is $b$-invariant,
since $b(G)=0$. De Carvalho, Lucchesi and Murty \cite{CLM0271} showed that  each removable edge is also $b$-invariant in a solid brick, where a solid brick is a brick free of non-trivial separating cuts.  Confirming  a conjecture of Lov\'asz,
 they  proved  in \cite{CLM0270} that
every brick, distinct from $K_4$, $\overline{C}_6$ and
the Petersen graph, has a $b$-invariant edge.
 De Carvalho, Lucchesi and Murty  \cite{CLM1216}  showed that every solid brick $G$ has at least $|V(G)|/2$ $b$-invariant edges.
  Based on properties of $b$-invariant edges, all bricks can be generated by
using  several operations from three basic bricks \cite{CLM0630}. $b$-invariant edges have many applications in
matching theory.
Readers may refer to \cite{CLM0293,KM1652,M0400} for applications of $b$-invariant edges.

Let $k$ be a positive integer. Recall that a graph $G$ is
{\it $k$-edge-connected} if $|C| \ge k$ for every edge cut
$C$ of $G$. An
edge cut with $k$ edges is called a {\it $k$-cut}. A cubic graph is {\it essentially 4-edge-connected} if it
is 2-edge-connected and   the only 3-cuts are the trivial
ones. Recently,  Kothari, de Carvalho, Lucchesi  and Little considered the
property of removable edges in essentially 4-edge-connected cubic bricks,  and
  showed that each
essentially 4-edge-connected cubic non-near-bipartite brick $G$,
distinct from the Petersen graph, has at least $|V(G)|$ $b$-invariant
edges \cite{KCLL19a}.  They also made the following
conjecture in the same paper.
\begin{conj}{\em \cite{KCLL19a}}\label{conj}
Every essentially 4-edge-connected cubic near-bipartite brick $G$,
distinct from $K_4$, has at least $|V(G)|/2$ $b$-invariant edges.
\end{conj}

Denote by $H_k$ the Cartesian product of a path of order ${k}$($k\geq2$)   and $K_2$ (the complete graph with two vertices). Suppose the four vertices with degree two of
$H_k$ are $\{u,v,x,y\}$ such that $u$ and $x$ lie in the same color class of $ H_k$.  By adding edges  $ux,vy$ to $ H_k$, we get a {\em prism} if $k$ is odd, and  a {\em M\"obius ladder}  if $k$ is even.  Prisms and M\"obius ladders are two types of cubic bricks  which play an important role in generating bricks \cite{CLM0630,No}. \footnote{By adding edges  $uy,vx$ to $ H_k$, we  also get a {\em prism} when $k$ is even, and  a {\em M\"obius ladder}  when $k$ is odd. In this case, the resulting graphs are bipartite, which are two types of cubic braces, see \cite{M0400} for example. }

Kothari, de Carvalho, Lucchesi and Little \cite{KCLL19a} also point out two infinite families that attain
the lower bound in Conjecture \ref{conj} exactly are: prisms of order $4k+2$, and M\"obius
ladders of order $4k$, where $k \geq  2$.
In this paper we present
a proof of Conjecture \ref{conj} and characterize all the graphs
that attain this lower bound.~The main result is stated as follows.

\begin{thm}\label{thm:conj}
Every essentially 4-edge-connected cubic near-bipartite
brick $G$, distinct from $K_4$, has at least $|V (G)|/2$
$b$-invariant edges. Furthermore, prisms of order $4k + 2$,
and M\"obius ladders of order $4k$, where $k \geq 2$, are
the only two families of graphs that attain this lower bound.
\end{thm}

The proof of Theorem \ref{thm:conj} will be
given in Section \ref{sec:pf} after we present some
properties concerning removable edges and removable doubletons
of  matching covered graphs in Section \ref{sec:pre}.
We now state two theorems that will be useful in the proof of the main result.

\begin{thm}{\em\cite{KCLL19a}}\label{ec}
In an essentially 4-edge-connected cubic brick, each edge is either
removable or otherwise participates in a removable doubleton. Moreover,
each removable edge is either $b$-invariant or otherwise quasi-$b$-invariant.
 \end{thm}

\begin{thm}{\em\cite{KCLL19a}}\label{cube}
Let $G$  be an essentially 4-edge-connected cubic near-bipartite brick that has two adjacent quasi-$b$-invariant edges. Then $G$ is the Cubeplex (see Fig.\ref{bn1} (a)).
 \end{thm}

\section{Equivalence classes in a brick}
\label{sec:pre}

Let $G$ be a matching covered graph. Two edges $e_1,e_2$ of
$G$ are {\it mutual dependence} if either $\{e_1,e_2\}\subseteq M$ or
$\{e_1,e_2\}\cap M=\emptyset$ for every perfect matching
$M$ of $G$. Obviously, mutual dependence is an equivalence relation. It partitions the edge
set $E(G)$ into equivalence classes\footnote{This terminology follows de Carvallo et al. \cite{CLM9914}. An equivalence
class contained at least two edges is called an equivalent set in \cite{HEYZ1956}.}. Lov\'asz proved the following attractive property of
 an equivalence class  in a brick.

\begin{thm}{\em~\cite{L8772}}\label{thm:lo}
Let $G$ be a brick and let $e$ and $f$ be two distinct mutual dependence edges of $G$. Then $G-\{e,f\}$ is bipartite.
\end{thm}
An equivalence class in a brick contains at most two edges; and
a
removable doubleton in a brick is an equivalence class by Theorem \ref{thm:lo}.
Obviously, the intersection of any two different equivalence classes of a brick
is the empty set. Two distinct equivalence classes of   a matching covered graph
are {\em mutually exclusive} if no perfect matching contains edges in both classes. The following theorem gave a characterization  of  bricks with more than two mutually exclusive removable doubletons.

\begin{thm}{\em\cite{CLM9914}}\label{thm:3-re}
If a brick $G$ has three mutually exclusive removable doubletons,
then either $G$ is $K_4$ or $G$  is $\overline{C}_6$,
up to multiple edges joining vertices of both triangles of $\overline{C}_6$.

\end{thm}

An ear decomposition of a matching covered graph $G$ is {\em optimal} if,
among all ear decompositions of $G$, it uses the least possible number of
double ears (for  the definition of  the ear decomposition, see page 174 in  \cite{LP86}).

\begin{thm}{\em\cite{CLM0293}}\label{thm:ear-n}
The number of double ears in an optimal ear decomposition
of a matching covered graph $G$ is $b(G)+p(G)$, where
$p(G)$ is the number of bricks of G whose  are isomorphic to the Petersen graph (up to multiple edges).
\end{thm}

It can be checked that the number of double ears in an optimal ear decomposition
of a near-bipartite graph is 1.
The following proposition is implied by Theorem \ref{thm:ear-n}.

\begin{pro}\label{nbn}
If $G$ is a near-bipartite graph, then $b(G)=1$.
\end{pro}

We say a bipartite graph $G(A,B)$ is  {\em balanced} if $|A|=|B|$. A
  bipartite graph with a perfect matching is always balanced. For the equivalence classes in a bipartite graph, we have the following result, which is an immediate consequence of Lemma 4.1 in \cite{CLM9914}.

\begin{lem}\label{thm:bi-eq}
Suppose  $e_1$ and $e_2$ are equivalent in a
 matching covered bipartite graph $G$. Then  $\{e_1,e_2\}$ forms a 2-edge-cut
which separates $G$ into two balanced components.
%
\end{lem}

The following Dulmage-Mendelsohn decomposition of a bipartite graph with a perfect matching will be used later.

\begin{thm}{\em \cite{MD1958}}\label{dm}
Let $G(A,B)$ be a bipartite graph with a perfect matching.
An edge $e$ of $G$ does not lie in any perfect matching of $G$
 if, and only if, there exists a
partition $(A_1, A_2)$ of $A$ and a partition $(B_1, B_2$) of $B$ such that $|A_1|=|B_1|$,
$e\in E_G(A_2, B_1 )$ and $E_G(A_1, B_2 )=\emptyset$.
\end{thm}

An equivalence class with two edges in a brick is not always a removable doubleton, that is
a brick contained  two distinct equivalent edges could be non-near-bipartite.
For example,  it can be checked that $\{ e_1,e_2\}$ is the only    equivalence class with two edges in  the graph in Figure \ref{bn1} (b); after removing $e_1$ and $e_2$, no perfect matching in the remaining graph would contain  any red edge. So this graph is not near-bipartite.
But for cubic bricks, this result is true as shown in the following proposition.

\begin{figure}[htbp]
 \centering
 \begin{minipage}{6cm}
\includegraphics[width=4.5cm]{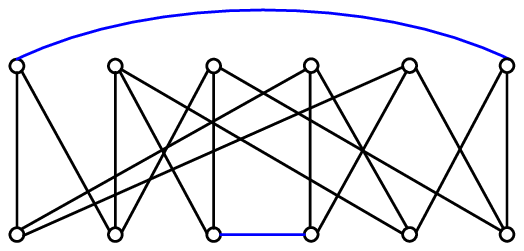}\\\centering(a)
\end{minipage}
\begin{minipage}{4cm}
\includegraphics[width=2.7cm]{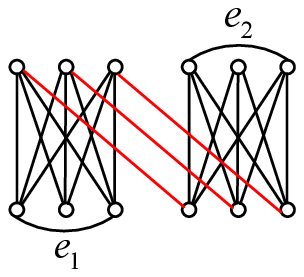}\\\centering(b)
\end{minipage}

  \caption{(a)~ The Cubeplex; (b)~a non-near-bipartite brick with two distinct equivalent edges.}\label{bn1}
\end{figure}

\begin{pro}\label{pro:3con}
Let $G$ be a cubic brick. If there exist two edges $e $ and $f$ such
that $G-\{e,f\}$ is bipartite, then $G$ is near-bipartite.
\end{pro}

\begin{proof}
Since $G$ is a brick, $G$ is not bipartite and is matching covered.
  By the definition of a near-bipartite graph, we need to show that $G-\{e,f\}$ is matching covered to complete the proof.   Suppose the two color classes of $G-\{e,f\}$ are $A$
  and $B$.
  Obviously, two ends of $e$ lie in the same color class of $G-\{e,f\}$, say $A$.  So two ends of $f$ belong to $B$.
  Suppose to the contrary that $G-\{e,f\}$ is not matching covered. By Theorem \ref{dm},
$A$
  and $B$ can be partitioned into $A_i$ and $B_i$ $(i=1,2)$, respectively,   such that $|A_1|=|B_1|$,
$|E_{G-\{e,f\}}(A_1,B_{2})|\geq 1$  and $|E_{G-\{e,f\}}(A_2,B_{1})|=0$.  Recalling that $G$ is cubic, by    calculating the sum of degrees of the vertices in $A_i$ and $B_i$, we have either $E_G(A_1,A_2)=\{e\}$, $|E_G(A_1,B_2)|=1$, $E_G(B_1,A_2\cup B_2)=\emptyset$ or $E_G(B_1,B_2)=\{f\}$, $|E_G(A_1,B_2)|=1$, $E_G(A_2,A_1\cup B_1)=\emptyset$.
we have
$|E_G(A_1\cup B_1,A_2\cup B_{2})|=2$ in both cases. Therefore $G$ is not 3-connected, contradicting   the fact that $G$ is a brick.
\end{proof}

The next two propositions concern the equivalence classes in a near-bipartite cubic brick.
The first one was already known to de Carvalho, Kothari, Lucchesi and Murty in 2014.

\begin{pro}\label{pro:4econ}
Suppose $\mathscr{E}$ is a set of removable doubletons of an  essentially 4-edge-connected cubic brick $G$ and $|\mathscr{E}|\geq 2$.
 Then $G$ can be decomposed into balanced bipartite vertex-induced subgraphs $G_i$ $(i=1,2,\ldots,|\mathscr{E}|)$ satisfying
  $E_G(G_j,G_{k})$   is a  removable doubleton of $G$ that belongs to $\mathscr{E}$ if $|j-k|\equiv 1\pmod {|\mathscr{E}|}$; otherwise, $E_G(G_j,G_{k})=\emptyset$.
Furthermore,  if $\mathscr{E}$ is the set of  all the removable doubletons in  $G$,   then
 $|V(G_i)|\neq 4$ for every $i\in \{1,2,\ldots,|\mathscr{E}|\}$. \\
\end{pro}

\begin{proof}   Suppose $\mathscr{E}:=\{\{e_i, e_{i}'\}:  i=1,2,\ldots,s\}$, where $s=|\mathscr{E}|$. Let $G':=G- \{e_1,e_1'\}$.
 By Theorem \ref{thm:lo}, $G'$ is bipartite.   By Proposition \ref{pro:3con},  $G$ is near-bipartite.  So
 $G'$ is matching covered.
  Note that every perfect matching of $G'$ is also a perfect matching of $G$ and $\{e_i,e_i'\}$ is a  removable doubleton  of $G$.
  Then, for  $i=2,3,\ldots,s$, $ e_i $ and $ e_i' $ are   equivalent in $G'$. By Lemma \ref{thm:bi-eq}, for  $i=2,3,\ldots,s$, $\{e_i,e_i'\}$ is an edge-cut separating $G'$ into two balanced components, denoted by $H_i$ and $H_i'$.

  {\bf Claim.} For any $j,k \in  \{2,3,\ldots,s\}(j\neq k)$, $V(H_j)\subset V(H_k)$ or $V(H_j)\subset V(H'_k)$.

  {\em Proof of Claim.} Suppose to the contrary that each of the four sets  $ V(H_j)\cap V(H_k)$, $ V(H'_j)\cap V(H_k)$, $ V(H'_j)\cap V(H'_k)$ and $ V(H_j)\cap V(H'_k)$ is not empty.
  Let $U_1:= V(H_j)\cap V(H_k)$, $U_2:=  V(H'_j)\cap V(H_k)$, $U_3:=  V(H'_j)\cap V(H'_k)$ and $ U_4:= V(H_j)\cap V(H'_k)$.

   Noticing both $|V(H_j)|$ and $|V(H_k)|$ are even, the parities of the cardinalities of  those  four sets  $ U_1, U_2,U_3$ and $U_4$ are the same.
   Recalling that $G'$ is matching covered, it is 2-connected. Therefore,  $ |E_{G'}(U_i,\overline{U_i})| \geq 2$ for $i=1,2,3,4$.
  Noting that  $|E_{G'}(V(H_j),V(\overline{H}_j))|=2=|E_{G'}(V(H_k),V(\overline{H}_k))|$, we have  $ |E_{G'}(U_i,\overline{U_i})|=2$ for $i=1,2,3,4$.  Note that  an edge in $G$    contributes at most 2 in $ \sum_{i=1}^4|E_G(U_i,\overline{U_i})| $. Therefore, $ \sum_{i=1}^4|E_{G}(U_i,\overline{U_i})|\leq \sum_{i=1}^4|E_{G'}(U_i,\overline{U_i})|+2\times 2= 12$.

  On the other hand,  we have $\sum_{i=1}^4|E_{G}(U_i,\overline{U_i})|\geq 16$.
  If $ |U_1|$ is even, then $|E_G(U_i,\overline{U_i})|\geq 4$ since $G$ is essentially 4-edge-connected for $i=1,2,3,4$, therefore,  $\sum_{i=1}^4|E_{G}(U_i,\overline{U_i})|\geq 16$.
  Now we consider when
 $ |V(H_j)\cap V(H_k)|$ is odd. If three vertex sets in $U_1,U_2,U_3$ and $U_4$ are singleton, without loss of generality, suppose $|U_1|=|U_2|=|U_3|=1$. Then $|E_{G}(U_4,\overline{U_4})|=3$ since $G$ is cubic.  A contradiction with the fact that $G$ is essentially 4-edge-connected. So  at most two  vertex sets in $U_1,U_2,U_3$ and $U_4$ are singleton. Again, by $G$ is cubic, $\sum_{i=1}^4|E_{G}(U_i,\overline{U_i})|\geq 3+3+5+5=16.$
 This contradiction shows the Claim holds.
   $\hfill\square$

Now, we can choose $i_1\in \{2,3,\ldots, s \}$ such that $V(H_{i_1})\subset V(H_{i})$ for every $i\in \{2,3,\ldots, s \}\setminus \{i_1\}$; for $j=2,3,\ldots, s-1$, let $i_j\in \{2,3,\ldots, s \}\setminus \{i_1,i_2,\ldots, i_{j-1}\}$ such that $V(H_{i_j})\subset V(H_{i}) $ for every $ i\in \{2,3,\ldots, s \}\setminus \{i_1,i_2,\ldots, i_{j}\}$ (exchange the notation of $H_i$ with $H_i'$ if needed). Let  $G_1=H_{i_1}$, $G_j=H_{i_{j}}-V(H_{i_{j-1}})$ for $j=2,3,\ldots, s-1$ and $G_s=H'_{i_{s-1}}$. Again, since $G$ is essentially 4-edge-connected, $E_G(G_1, G_s)=\{ e_1,e_1' \}$. Therefore $G_i(i=1,2,\dots, s)$ is the subgraph we need.

If $\mathscr{E}$ is the set of  all the  removable doubletons in  $G$,  suppose to the contrary that there exists some $i$ such that  $|V(G_i)|= 4$ (assume that $ A_i$ and $ B_i$ are the color classes of $G_i$). Then $G_i$ is a 4-cycle, denoted by $a_1b_1a_2b_2a_1$, where $a_1,a_2\in A_i$ and $b_1,b_2\in B_i$. Therefore $G-\{a_1b_2,a_2b_1\}$ is a bipartite graph with color classes $\cup_{j=1}^{i-1}A_j\cup \{a_1,b_2\}\cup(\cup_{j=i+1}^{s}B_j )$ and $\cup_{j=1}^{i-1}B_j\cup \{b_1,a_2\}\cup( \cup_{j=i+1}^{s}A_j) $. By Proposition \ref {pro:3con}, $\{a_1b_2, a_2b_1\}$ constitutes  a removable doubleton of $G$ which is not in $\mathscr{E}$. A contradiction. So the result follows.
\end{proof}

\begin{pro}\label{lem:ct}
Suppose  $\{e_1,e_1'\}$ and $\{e_2,e_2'\}$ are removable doubletons of a cubic brick $G$.
If both $e_1$ and $e_2$ are incident with $v_0$, then $e_1'$ and $e_2'$ are adjacent, and $v_0u_0\in E(G)$, where   $u_0$ is the common vertex of $e_1'$ and $e_2'$.
\end{pro}

\begin{proof} Suppose $e_1'=u_0u_1, e_2'=u'_0u_2$ and $N(v_0)=\{v_1,v_2, v_3\}$, where $e_1=v_0v_1, e_2=v_0v_2$. We will show that $u_0=u_0'=v_3$ to complete the proof.
By Proposition \ref{pro:4econ}, $G$ can be decomposed into    balanced bipartite vertex-induced subgraphs $G_i(A_i,B_i)~(i=1,2)$  such that $v_0\in A_1$, $u_0,u_0',v_3\in B_1$, $v_1,u_2\in A_2$, $v_2,u_1\in B_2$.

Note that $N(A_1-\{v_0\}) =B_1$. Then simple counting argument shows that $V(G_1)-\{v_0\}$ is a shore of a tight
cut of $G$. Since $G$ is free of nontrivial tight cuts, $V(G_1)-\{v_0\}$ must be trivial. We
deduce that $B_1 = \{u_0\}$ and $A_1 = \{v_0\}$. This proves Proposition \ref{lem:ct}.

\end{proof}

\section{The proof of the main theorem}
\label{sec:pf}
It is easy to check that for a cubic brick $G$, $G$ is isomorphic
to $K_4$ if $|V(G)|=4$, and $G$ is isomorphic to $\overline{C_6}$
if $|V(G)|=6$. Noticing $\overline{C_6}$ is not essentially 4-edge-connected, in the following, we always assume that $|V(G)|\ge8$.
We can classify the edges of $G$, by Theorem \ref{ec},
into three disjoint classes: (1) edges that participate in a removable
doubleton, (2) $b$-invariant edges, and (3) quasi-$b$-invariant edges. For
simplicity, we denote the three edge sets by $E_1$, $E_2$ and $E_3$,
respectively.  Therefore,  $|E_1|+|E_2|+|E_3|=\frac{3}{2}|V(G)|$.
We will show that $|E_2|\geq |V(G
)|/2$ to complete the proof.
Note that the
Cubeplex contains $14>6=|V (G)|/2$ $b$-invariant edges. So we suppose $G$ is
not the Cubeplex. Therefore, by Theorem \ref{cube}, every vertex
in $G$ is incident with at most one quasi-$b$-invariant edge, that is  $|E_3|\leq |V(G
)|/2$. We will
consider the following two cases depending on the number of removable
doubletons.

Case 1. $G$ has at most two removable doubletons.

This implies that $|E_1|\leq 4$. Recalling
that $|V(G)|\ge8$,   $|E(G)|\ge12$. So
$|E_1|\leq |V (G)|/2$. Recall that $|E_3|\leq |V(G
)|/2$. Hence,
$|E_2|\geq |V(G
)|/2$.

Now, we show every brick has more than $|V(G
)|/2$ $b$-invariant edges in
this case. Firstly, we claim that $|V(G)|=8$ and
$|E_1|=4$. Otherwise, $|V(G)|>8$, or $|E_1|=2<|V (G)|/2$.
Since $|E_1|\leq 4$, we have
$|E_1|<|V (G)|/2$ in either case. Since $|E_3|\le |V (G)|/2$, it follows that
$|E_2|> |V (G)|/2$. Therefore, $G$ contains
more than $|V (G)|/2$ $b$-invariant edges, a contradiction. Thus,
$|V(G)|=8$ and $|E_1|=4$.

\begin{figure}[htbp]
 \centering
  \includegraphics[width=4cm]{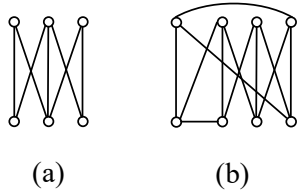}\\
  \caption{(a) the graph isomorphic to $G_2$; (b) $G'$.}\label{bl}
\end{figure}

As $|V(G)|=8$ and $|E_1|=4$,
by Proposition \ref{pro:4econ}, we may assume that $G-E_1$ contains two
bipartite vertex-induced subgraphs $G_1$ and $G_2$, and $|V(G_1)|=2$ and $|V(G_2)|=6$. Then, $G_1$ is
isomorphic to $K_2$, $G_2$ contains
four vertices with degree two and the remaining two vertices have degree
three. It is easy to check that $G_2$ is isomorphic the graph in
Figure \ref{bl} (a). Recall that $G$ is near-bipartite. Hence, $G$ is
isomorphic to the M\"obius ladder
with 8 vertices or the graph $G'$ shown in Figure \ref{bl} (b). However, the M\"obius ladder
with 8 vertices has four distinct removable doubletons, and $G'$ contains
a triangle which implies that it contains a nontrivial 3-cut. So $G'$ is not essentially 4-edge-connected,  giving a
contradiction.  Therefore, no graph has at most two
removable doubletons and exactly $|V(G)|/2$ $b$-invariant edges.\\

Case 2. $G$ has more than two removable doubletons.

We will show that each vertex in $G$ is incident with at
least one $b$-invariant edge. Recall that every vertex
in $G$ is incident with at most one quasi-$b$-invariant edge. So,
it is enough to show the following claim.

{\bf Claim 1.} If  every edge in $\{uu_1, uu_2\}$ participates in a
removable doubleton, then  $uv$ is
$b$-invariant in $G$, where $v\in N(u)-\{u_1,u_2\}$.

{\em Proof of Claim 1.} Firstly, we claim $uv$ is removable in $G$. Otherwise, $G$ contains three mutually
exclusive removable doubletons, so that either $G$ is $K_4$ or $G$  is $\overline{C}_6$,
up to multiple edges joining vertices of both triangles of $\overline{C}_6$ by Theorem \ref{thm:3-re}, contradicting  the hypothesis that $|V(G)|\geq 8$.

Note that $G$ has more than two removable doubletons, we may assume that
$\{e,e'\}$ is a removable doubleton of $G$ such that $\{e,e'\}\cap\{uu_1,uu_2\}=\emptyset $. Now, we will show that
$uv$ is also removable in $G-\{e,e'\}$. Assume to the contrary that there
exists an edge $f$ of $G$ that is not contained in any perfect matching of
$G-\{e,e',uv\}$. Note that $G-\{e,e'\}$ is matching covered  by Proposition \ref{pro:3con}. Then any perfect matching $M_1$ of $G-\{e,e'\}$ that
contains $f$   also contains
$uv$. By Proposition \ref{lem:ct}, we may assume that $\{uu_1,vv_1\}$ and
$\{uu_2,vv_2\}$ are two removable doubletons of $G$. By Proposition
\ref{pro:4econ}, we may assume $G$ can be decomposed
into balanced bipartite vertex-induced subgraphs $G_i(A_i,B_i)(i=1,2,3)$ satisfying: \\
(1) $E(A_1,B_2)=vv_1$,  $E(A_2,B_3)=uu_2$, $E(B_1,A_2)=uu_1$,
$E(B_2,A_3)=vv_2$, $E(A_1,A_3)=e$ and $E(B_1,B_3)=e'$;\\
(2) $G_2=uv$ (see Figure \ref{bn}).

\begin{figure}[htbp]
 \centering
  \includegraphics[width=5cm]{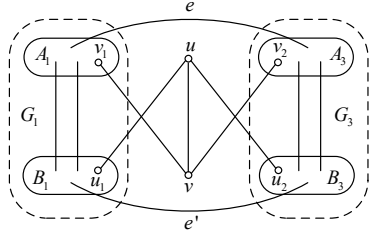}\\
  \caption{illustration in the proof of Claim 1.}\label{bn}
\end{figure}

Assume without loss of generality that $f\in E(G_1)$. Then
$M_1\cap E(G_1)$ is a perfect matching of $G_1$ that contains $f$. Let
$M_2$ be an arbitrary perfect matching of $G-\{e,e'\}$ that contains the
removable doubleton of $\{uu_2,vv_2\}$. Then $M_2\cap E(G-V(G_1))$ is a
perfect matching of $G-V(G_1)$. So, $(M_1\cap E(G_1))\cup (M_2\cap E(G-V(G_1))$ is a perfect matching
of $G-\{e,e'\}$ that contains $f$ but not $uv$, giving a contradiction.

Finally, since $uv$ is removable in both $G$ and $G-\{e,e'\}$, we conclude
that $G-uv$ is   near-bipartite   and $\{e,e'\}$ is a removable doubleton
of $G-uv$. So, $b(G-uv)=1$ by Proposition \ref{nbn}. Therefore, $uv$ is
$b$-invariant in $G$.
 $\hfill\square$

Now, we show that all the graphs that attain this lower bound are
a prism  of order $4k + 2$, and a M\"obius ladder  of order  $4k$,
for some $k \geq  2$.  Now we suppose
that $G$ is an essentially 4-edge-connected cubic near-bipartite
brick with exactly $|V (G)|/2$
$b$-invariant edges,  so that $|E_1|>4$. We have the following claim.

{\bf Claim 2.} Each component of $G-E_1$ is $K_2$.

{\em Proof of  Claim 2.}
Suppose  to the contrary that  there exists
a component of $G-E_1$, say $G_i$,  satisfying $|V(G_i)|> 2$. Then $|V(G_i)|\ge6$ by Proposition \ref{pro:4econ}. Now we  consider the edge
set $E(G_i)$. Note that $G_i$ contains at most $|V (G_i)|/2$
quasi-$b$-invariant edges of $G$ by Theorem \ref{cube}.  Since each edge of $E(G_i)$ is removable in $G$,
$G_i$ contains at least $|E(G_i)|-\frac{|V(G_i)|}{2}$ $b$-invariant edges of $G$.
By Proposition \ref{pro:4econ}, $G_i$ and $G-V(G_i)$  are connected by two pairs of removable doubletons of $G$. Since $|V(G_i)|> 2$, every edge in one of those removable doubletons is not adjacent to any edge in the other removable doubleton by Proposition \ref{lem:ct}.
Therefore, except four vertices, the degrees of  all the other vertices in $G_i$ are  three,
that is $|E(G_i)|=3|V (G_i)|/2-2$.
 This
implies that $G_i$ contains more than $|V (G_i)|/2$ $b$-invariant edges. For every
  component $G_j$   with two vertices,   both of those two vertices are incident with
two edges which lie in different removable doubletons, respectively. By Claim 1,  the unique edge of the component
is $b$-invariant. So $G_j$ contains exactly $|V (G_j)|/2(=1)$ $b$-invariant
edges of $G$. Therefore, we can conclude that $G$ contains more than $|V (G)|/2$
$b$-invariant edges if $G-E_1$ contains a component with more than one edge,
giving a contradiction.
 $\hfill\square$

 Each vertex of $G$ is incident with two edges in $E_1$ by Claim 2. Hence, $G$ is
isomorphic to a prism if $|G|=4k + 2$, and is isomorphic to a  M\"obius
ladder  if $|G|=4k$, for some $k\geq 2$. Therefore, Theorem \ref{thm:conj} holds.

\section{Acknowledgements}

The authors would like to express their gratitude to the referees,
for their careful reading and valuable
comments towards the improvement of the paper, and also for giving the present   proof of
Proposition \ref{lem:ct}.

\bibliographystyle{plain}
\bibliography{references}






\bibliographystyle{amsplain}

\end{document}